\documentclass{article}

\usepackage{amsmath,amsthm,bm,colortbl,graphicx,mathrsfs,amssymb,afterpage,url,cite,verbatim,ulem}
\theoremstyle{plain}
\newtheorem{theorem}{Theorem}

\theoremstyle{remark} 

\def\P{{\rm P}} 
\def\E{{\rm E}} 
\allowdisplaybreaks

\makeatother
\title{The flashing Brownian ratchet \\ and Parrondo's paradox}
\author{S. N. Ethier\thanks{Department of Mathematics, University of Utah, 155 S. 1400 E., Salt Lake City, UT 84112, USA. e-mail: ethier@math.utah.edu.}\; and Jiyeon Lee\thanks{Department of Statistics, Yeungnam University, 280 Daehak-Ro, Gyeongsan, Gyeongbuk 38541, South Korea.  e-mail: leejy@yu.ac.kr.}}
\date{}

\begin{document}

\maketitle

\begin{abstract}
A Brownian ratchet is a one-dimensional diffusion process that drifts toward a minimum of a periodic asymmetric sawtooth potential.  A flashing Brownian ratchet is a process that alternates between two regimes, a one-dimensional Brownian motion and a Brownian ratchet, producing directed motion.  These processes have been of interest to physicists and biologists for nearly 25 years.  The flashing Brownian ratchet is the process that motivated Parrondo's paradox, in which two fair games of chance, when alternated, produce a winning game.  Parrondo's games are relatively simple, being discrete in time and space.  The flashing Brownian ratchet is rather more complicated.  We show how one can study the latter process numerically using a random walk approximation.
\medskip\par

\end{abstract}

\section{Introduction}\label{intro}

The flashing Brownian ratchet was introduced by Ajdari and Prost \cite{AP92}; see also Magnasco \cite{M93}.  It is a stochastic process that alternates between two regimes, a one-dimensional Brownian motion and a Brownian ratchet, the latter being a one-dimensional diffusion process that drifts toward a minimum of a periodic asymmetric sawtooth potential.  The result is directed motion, as explained in Figure~\ref{HATP} (from Harmer et al. \cite{HATP01}) and Figure~\ref{Parrondo-Dinis} (from Parrondo and Din\'is \cite{PD04}).  Earlier versions of these figures appeared in Rousselet et al.\ \cite{RSAP94} and Faucheux et al.\ \cite{FBKL95}.  For another version see Amengual \cite[Fig.\ 2.3]{A06}.  

The flashing Brownian ratchet is of interest not just to physicists but also to biologists in connection with so-called molecular motors; see e.g.\ Bressloff \cite[Chap.\ 4]{B14}.  The flashing Brownian ratchet is the process that motivated Parrondo's paradox (Harmer and Abbott \cite{HA99,HA02}), in which two fair games of chance, when alternated, produce a winning game.  

\begin{figure}[!h]
\centering
\includegraphics[width=2.8in]{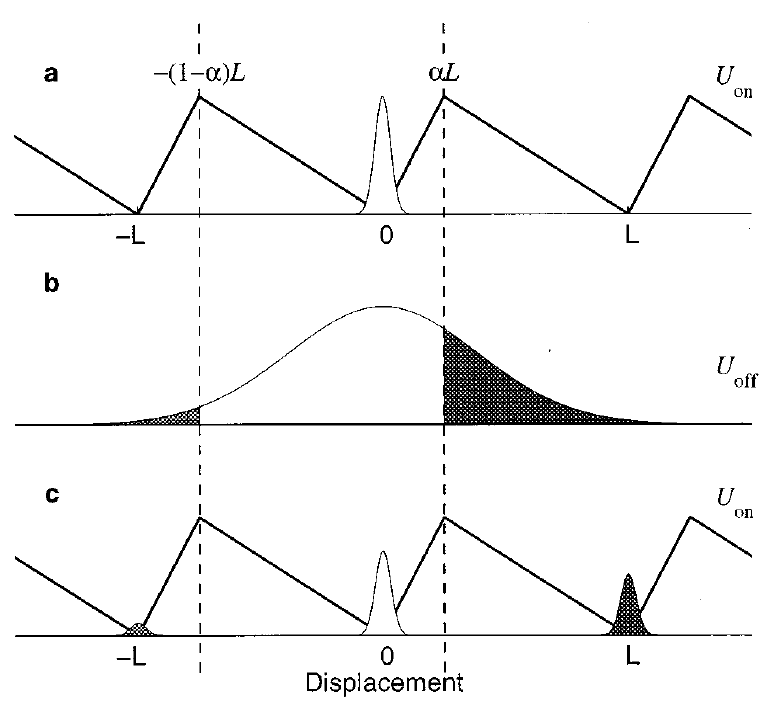}
\caption{Brownian ratchet mechanism. The sawtooth and flat potentials are labeled with $U_{\text{on}}$ and $U_{\text{off}}$, respectively, while the distribution of Brownian particles is shown via the normal curves. This sequence of flashing between on and off potentials shows there is a net movement of particles to the right. [Reprinted from Harmer et al. \cite{HATP01} with the permission of AIP Publishing.]}
\label{HATP}
\end{figure}

\begin{figure}[!h]
\centering
\includegraphics[width=3in]{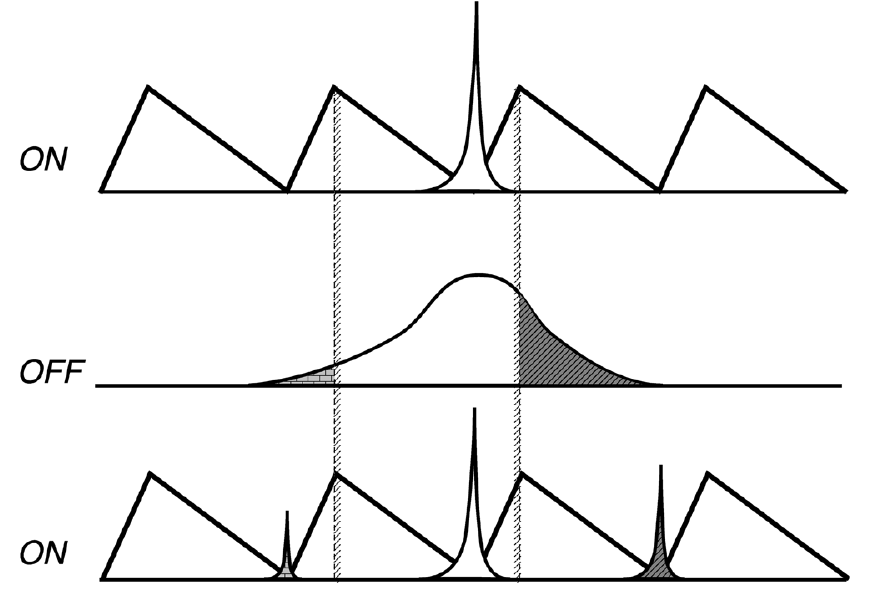}
\caption{\label{Parrondo-Dinis}The flashing ratchet at work.  The figure represents three snapshots of the potential and the density of particles. Initially (upper figure), the potential is on and all the particles are located around one of the minima of the potential. Then the potential is switched off and the particles diffuse freely, as shown in the centred figure, which is a snapshot of the system immediately before the potential is switched on again. Once the potential is connected again, the particles in the darker region move to the right-hand minimum whereas those within the small gray region move to the left. Due to the asymmetry of the potential, the ensemble of particles move, on average, to the right.  [Reprinted from Parrondo and Din\'is \cite{PD04} with the permission of Taylor \& Francis Ltd.]}
\end{figure}

Our aim here is to show, via a precise mathematical formulation of the flashing Brownian ratchet, how one can study the process numerically using a random walk approximation.  In Section~\ref{Parrondo-games} we provide a new formulation of Parrondo's paradox motivated by the flashing Brownian ratchet.  These Parrondo games are then modified in Section~\ref{approx} so as to yield our random walk approximation.  We determine, in Sections~\ref{modeling1} and \ref{modeling2}, whether the conceptual Figures~\ref{HATP} and \ref{Parrondo-Dinis} accurately represent the behavior of the flashing Brownian ratchet.  

Alternatively, one could numerically solve a partial differential equation, specifically the Fokker--Planck equation, to obtain similar results, but we believe that our method is simpler.  Discretization of the Fokker--Planck equation for the Brownian ratchet, and the relationship to Parrondo's games, has been explored by Allison and Abbott \cite{AA02} and Toral et al. \cite{TAM03a,TAM03b}.

Using the notation of Figure~\ref{HATP}, it is clear how to formulate the model.  First, the asymmetric sawtooth potential $V$ is given by the formula
\begin{equation}\label{sawtooth}
V(x):=\begin{cases} x/\alpha&\text{if $0\le x\le\alpha L$,}\\(L-x)/(1-\alpha)&\text{if $\alpha L\le x\le L$,}\end{cases}
\end{equation}
extended periodically (with period $L$) to all of ${\bf R}$.  Here $0<\alpha<1$ and $L>0$, and asymmetry requires only that $\alpha\ne1/2$.  
($L$ is a scale factor that is not important; some authors take $L=1$.) The \textit{Brownian ratchet} is a one-dimensional diffusion process with diffusion coefficient 1 and drift coefficient $\mu$ proportional to $-V'$, that is, for some $\gamma>0$,
$$
\mu(x):=-\gamma V'(x)=\begin{cases}-\gamma/\alpha&\text{if $0\le x<\alpha L$,}\\ \gamma/(1-\alpha)&\text{if $\alpha L\le x<L$,}\end{cases}
$$
again extended periodically (with period $L$) to all of ${\bf R}$.  Such a process $X_t$ is governed by the stochastic differential equation (SDE)
\begin{equation}\label{SDE1}
dX_t=dB_t+\mu(X_t)\,dt,
\end{equation}
where $B_t$ is a standard Brownian motion.  This diffusion process drifts to the left on $[nL,(n+\alpha) L)$ and drifts to the right on $[(n+\alpha) L,(n+1)L)$, for each $n\in{\bf Z}$.  In other words, it drifts toward a minimum of the sawtooth potential $V$.

Given $\tau_1,\tau_2>0$, the \textit{flashing Brownian ratchet} is a time-inhomogeneous one-dimensional diffusion process that evolves as a Brownian motion on $[0,\tau_1]$ (potential ``off''), then as a Brownian ratchet on $[\tau_1,\tau_1+\tau_2]$ (potential ``on''), then as a Brownian motion on $[\tau_1+\tau_2,2\tau_1+\tau_2]$ (potential ``off''), then as a Brownian ratchet on $[2\tau_1+\tau_2,2\tau_1+2\tau_2]$ (potential ``on''), and so on.  Such a process $Y_t$ is governed by the SDE
$$dY_t=dB_t+\eta(t)\mu(Y_t)\,dt,$$
where\footnote{By definition, $\text{mod}(t,\tau)$ is the remainder (in $[0,\tau)$) when $t$ is divided by $\tau$.}
$$
\eta(t):=\begin{cases} 0&\text{if mod$(t,\tau_1+\tau_2)<\tau_1$,}\\
                       1&\text{if mod$(t,\tau_1+\tau_2)\ge\tau_1$.}\end{cases}
$$
Notice that, once the parameters of the sawtooth potential ($\alpha$ and $L$) are specified, the flashing Brownian ratchet is specified by three parameters, $\gamma$, $\tau_1$, and $\tau_2$.  (Alternatively, we could let the diffusion coefficients of the Brownian motion and the Brownian ratchet be $\sigma_1^2$ and $\sigma_2^2$, respectively, instead of 1 and 1, and then take $\tau_1=\tau_2=1$.)  Our formulation is equivalent to that of Din\'is \cite[Eq.\ (1.78)]{D05}, though parameterized differently.

Occasionally we may want to wrap these processes (the Brownian ratchet and the flashing Brownian ratchet) around the circle of circumference $L$.  Because they are spatially periodic with period $L$, the wrapped processes remain Markovian.  For example, we could define the wrapped Brownian ratchet $\bar X_t$ by
$$
\bar X_t:=e^{(2\pi i/L)X_t}.
$$
Alternatively, we could simply define it as the $[0,L)$-valued process
$$
\bar X_t:=\text{mod}(X_t,L),
$$
with the understanding that the endpoints of the interval $[0,L)$ are identified, effectively making it a circle of circumference $L$.  The same procedure applies to the flashing Brownian ratchet, yielding
$$
\bar Y_t:=\text{mod}(Y_t,L).
$$

\section{Parrondo games from Brownian ratchets}\label{Parrondo-games}

We first consider the periodic drift coefficient $\mu$ described above in the case in which $\alpha=1/3$ and $L=3$.  We want to discretize space and time.  We replace each interval $[j,j+1)$ by its midpoint $j+1/2$, which we relabel as $j$.  In terms of $\mu$ we define the discrete drift by $\mu_j:=\mu(j+1/2)$.  This discretizes space, now interpreted as profit in a game of chance instead of displacement.  When the potential is off, we replace the Brownian motion by a simple symmetric random walk on {\bf Z} and call this game $A$, a fair coin-tossing game.  When the potential is on, we replace the Brownian ratchet by an asymmetric random walk on {\bf Z} whose periodic state-dependent transition probabilities are determined by the discrete drift and call this game $B$.  

\begin{figure}[htb]
\centering
\includegraphics[width=3.7in]{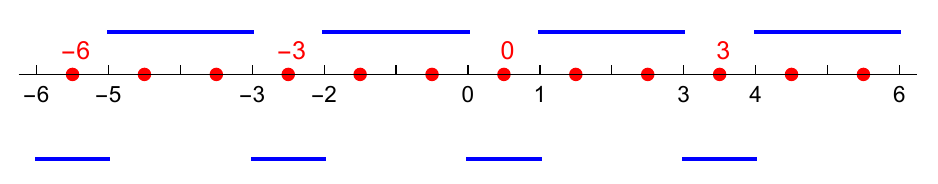}
\caption{\label{drift-fig}The periodic drift $\mu$ with $\alpha=1/3$ and $L=3$ is plotted on the interval $[-6,6]$.  Each interval $[j,j+1)$ (in black) is replaced by its midpoint $j+1/2$, which we relabel as $j$ (in red) to discretize space. To discretize time as well, we replace the  Brownian motion by a simple symmetric random walk on {\bf Z}, and we replace the Brownian ratchet by an asymmetric random walk on {\bf Z} whose periodic state-dependent transition probabilities are determined by a discretized version of $\mu$.}
\end{figure}

We find that the asymmetric random walk on {\bf Z} has periodic state-dependent transition probabilities of the form  
\begin{equation}\label{transitions1}
P(j,j+1):=\begin{cases}p_0&\text{if mod$(j,3)=0$,}\\ p_1&\text{if mod$(j,3)=\text{1 or 2}$,}\end{cases}
\end{equation}
and $P(j,j-1)=1-P(j,j+1)$, where $0<p_0<1/2$ since $\mu_j<0$ if $\text{mod}(j,3)=0$ and $1/2<p_1<1$ since $\mu_j>0$ if $\text{mod}(j,3)=1\text{ or }2$.  An invariant measure $\pi$, if it exists, must be periodic (i.e., $\pi(j)=\pi(j+3)$).  We can check that the detailed balance conditions 
\begin{equation}\label{detailed1}
\pi(0)p_0=\pi(1)(1-p_1),\quad \pi(1)p_1=\pi(2)(1-p_1),\quad \pi(2)p_1=\pi(0)(1-p_0)
\end{equation}
have a solution if and only if $(1-p_0)(1-p_1)^2=p_0p_1^2$.  Solving for $p_1$ in terms of $p_0$, we find that
$$
p_1=\frac{1}{1+\sqrt{p_0/(1-p_0)}}.  
$$
Denoting the square root in the denominator by $\rho$, the requirements that $0<p_0<1/2<p_1<1$ become $0<\rho<1$, and 
\begin{equation}\label{p0,p1-1}
p_0=\frac{\rho^2}{1+\rho^2},\qquad p_1=\frac{1}{1+\rho}.
\end{equation}
(This is the parameterization of Ethier and Lee \cite{EL09}.)  Further, in terms of $\rho$, the invariant measure obtained from \eqref{detailed1} has the form
$$
(\pi(0),\pi(1),\pi(2))=\frac{(1+\rho^2,\rho(1+\rho),1+\rho)}{2(1+\rho+\rho^2)},
$$
resulting in a mean profit of $\pi(0)(2p_0-1)+(\pi(1)+\pi(2))(2p_1-1)=0$, so game $B$ is also fair (asymptotically).  Nevertheless, the random mixture $c A+(1-c)B$ ($0<c<1$) is winning, as is the periodic pattern $A^r B^s$ for each $r,s\ge1$ except $r=s=1$.  This is the original form of \textit{Parrondo's paradox}.  The special case in which $\rho=1/3$, namely
$$
p_0=\frac{1}{10},\qquad p_1=\frac34,
$$
was the choice of Parrondo, at least in the absence of a bias parameter.

There are several proofs available for these results, including Pyke \cite{P03}, based on mod $m$ random walks, Key et al. \cite{KKA06}, based on random walks in periodic environments, Ethier and Lee \cite{EL09}, based on the strong-mixing central limit theorem, and R\'emillard and Vaillancourt \cite{RV15}, based on Oseledec's multiplicative ergodic theorem.

It should be mentioned that Pyke \cite{P03} found an elegant way to derive Parrondo's games \eqref{transitions1} from a one-dimensional diffusion process that can be interpreted as a Brownian ratchet but with the sawtooth potential having a shape different from~\eqref{sawtooth}.

The above formulation with $\alpha=1/3$ and $L=3$ can be generalized.  Let $0<\alpha<1$ and assume that $\alpha$ is rational, so that there exist relatively prime positive integers $l$ and $L$ with $\alpha=l/L$.  Game $A$ is as before, whereas game $B$ is described by an  asymmetric random walk on {\bf Z} with periodic state-dependent transition probabilities of the form  
\begin{equation}\label{transitions2}
P(j,j+1):=\begin{cases}p_0&\text{if mod$(j,L)<l$,}\\ p_1&\text{if mod$(j,L)\ge l$,}\end{cases}
\end{equation}
and $P(j,j-1)=1-P(j,j+1)$, where $0<p_0<1/2<p_1<1$ as before.  An invariant measure $\pi$, if it exists, must be periodic (i.e., $\pi(j)=\pi(j+L)$).  We can check that the detailed balance conditions 
\begin{align}\label{detailed2}
\pi(j)p_0&=\pi(j+1)(1-p_0),\quad 0\le j\le l-2,\nonumber\\
\pi(j)p_0&=\pi(j+1)(1-p_1),\quad j=l-1,\nonumber\\
\pi(j)p_1&=\pi(j+1)(1-p_1),\quad l\le j\le L-2,\nonumber\\
\pi(j)p_1&=\pi(0)(1-p_0),\quad j=L-1,
\end{align}
have a solution if and only if $(1-p_0)^l(1-p_1)^{L-l}=p_0^lp_1^{L-l}$.  Solving for $p_1$ in terms of $p_0$, we find that
\begin{align}
\label{p1}
p_1=\frac{1}{1+[p_0/(1-p_0)]^{l/(L-l)}}=\frac{1}{1+[p_0/(1-p_0)]^{\alpha/(1-\alpha)}}.  
\end{align}
Denoting the $\alpha/(1-\alpha)$th power in the denominator by $\rho$, the requirements that $0<p_0<1/2<p_1<1$ become $0<\rho<1$, and
\begin{equation}\label{p0,p1-2}
p_0=\frac{\rho^{(1-\alpha)/\alpha}}{1+\rho^{(1-\alpha)/\alpha}},\qquad p_1=\frac{1}{1+\rho}.
\end{equation}
Notice that \eqref{transitions2} and \eqref{p0,p1-2} generalize \eqref{transitions1} and \eqref{p0,p1-1}.

Further, in terms of $\rho$, the invariant measure obtained from \eqref{detailed2} is proportional to
\begin{align*}
\pi(j)&=\rho^{j(1-\alpha)/\alpha},\quad 0\le j\le l-1,\\
\pi(j)&=\rho^{L(1-\alpha)-(j-l+1)}\frac{1+\rho}{1+\rho^{(1-\alpha)/\alpha}},\quad l\le j\le L-1,
\end{align*}
resulting in a mean profit of 
\begin{align*}
&(\pi(0)+\cdots+\pi(l-1))(2p_0-1)+(\pi(l)+\cdots+\pi(L-1))(2p_1-1)\\
&\qquad{}=\frac{1-\rho^{L(1-\alpha)}}{1-\rho^{(1-\alpha)/\alpha}}\,\frac{\rho^{(1-\alpha)/\alpha}-1}{1+\rho^{(1-\alpha)/\alpha}}
+\rho^{L(1-\alpha)}\frac{\rho^{-1}-\rho^{-(L-l+1)}}{1-\rho^{-1}}\,\frac{1+\rho}{1+\rho^{(1-\alpha)/\alpha}}\,\frac{1-\rho}{1+\rho}\\
&\qquad{}=-\frac{1-\rho^{L(1-\alpha)}}{1+\rho^{(1-\alpha)/\alpha}}+\frac{1-\rho^{L(1-\alpha)}}{1+\rho^{(1-\alpha)/\alpha}}\\
&\qquad{}=0,
\end{align*}
so game $B$ is also fair (asymptotically). 

As a function of $p_0$ the function in (\ref{p1}) is strictly convex on $(0, 1/2)$ if $\alpha < 1/2$.  It follows that the random mixture $c A+(1-c)B$ $(0<c<1)$ has positive mean profit so that the Parrondo effect is present. If $\alpha > 1/2$, then the function in (\ref{p1}) is striclty concave on $(0, 1/2)$ and the anti-Parrondo effect, in which two fair games combine to lose, appears.

Thus, game $A$ and (the generalized) game $B$ lead to a new form of Parrondo's paradox.

\section{Approximating the Brownian ratchet}\label{approx}

As in Section~\ref{Parrondo-games}, let $0<\alpha<1$ and assume that $\alpha$ is rational, so that there exist relatively prime positive integers $l$ and $L$ with $\alpha=l/L$.  Consider a sequence of asymmetric random walks on ${\bf Z}$ with periodic state-dependent transition probabilities as follows.  Given $n\ge1$, we let
\begin{equation}\label{transitions3}
P_n(j,j+1):=\begin{cases}p_0&\text{if mod$(j,Ln)<ln$,}\\ p_1&\text{if mod$(j,Ln)\ge ln$,}\end{cases}
\end{equation}
and $P_n(j,j-1)=1-P_n(j,j+1)$, where
\begin{equation}\label{p0,p1}
p_0=\frac{\rho^{(1-\alpha)/\alpha}}{1+\rho^{(1-\alpha)/\alpha}},\qquad p_1=\frac{1}{1+\rho}
\end{equation}
as in \eqref{p0,p1-2}.  Notice that the special case of \eqref{transitions3} in which $n=1$ is precisely \eqref{transitions2}.

We want to let $n\to\infty$ but first we let 
\begin{equation}\label{rho}
\rho=1-\frac{\lambda}{n},
\end{equation}
where $\lambda>0$, then we rescale time by allowing $n^2$ jumps per unit of time, and finally we rescale space to $\{i/n:i\in{\bf Z}\}$ by dividing by $n$.  The result in the limit as $n\to\infty$ is a Brownian ratchet.  

\begin{theorem}\label{theorem1}
For $n=1,2,\ldots$, let $\{X_n(k),\,k=0,1,\ldots\}$ denote the random walk on ${\bf Z}$ defined by \eqref{transitions3}--\eqref{rho}, and let $X_t$ denote the Brownian ratchet with $\gamma=\lambda(1-\alpha)/2$.  If $X_n(0)/n$ converges in distribution to $X_0$ as $n\to\infty$, then $\{X_n(\lfloor n^2 t\rfloor)/n,\,t\ge0\}$ converges in distribution in $D_{\bf R}[0,\infty)$ to $\{X_t,\, t\ge0\}$ as $n\to\infty$.
\end{theorem}

\begin{proof}
The generator of the diffusion process satisfying the SDE \eqref{SDE1} is
$$
\mathscr{L}f(x):=\frac12 f''(x)+\mu(x) f'(x)
$$
acting on $C_c^\infty({\bf R})$, the space of real-valued $C^\infty$ functions on ${\bf R}$ with compact support, where 
$$
\mu(x):=\begin{cases}-\gamma/\alpha&\text{if $0\le x<\alpha L$,}\\ \gamma/(1-\alpha)&\text{if $\alpha L\le x<L$,}\end{cases}
$$
and $\mu$ is extended periodically (with period $L$) to all of ${\bf R}$.  Then, by virtue of the Girsanov transformation, the martingale problem for $\mathscr{L}$ is well posed (e.g., Stroock and Varadhan \cite[Theorem~6.4.3]{SV79}) and it suffices to show that the discrete generator $\mathscr{L}_n$, given by
$$
\mathscr{L}_n f(x):=n^2\{f(x+1/n)P_n(x,x+1)+f(x-1/n)P_n(x,x-1)-f(x)\},
$$
converges to $\mathscr{L}$ in the sense that
$$\lim_{n\to\infty}\sup_{x:nx\in {\bf Z}}|\mathscr{L}_n f(x)-\mathscr{L}f(x)|=0,\qquad f\in C^\infty_c({\bf R}).$$
Here we are using a result of Ethier and Kurtz \cite[Corollary~4.8.17]{EK86}.

If $\{\mu_n\}$ is a sequence of real numbers converging to $\mu$, then
\begin{align*}
&n^2\bigg[f\bigg(\! x+\frac1n\bigg)\frac12\bigg(1+\frac{\mu_n}{n}\bigg)+f\bigg(\! x-\frac1n\bigg)\frac12\bigg(1-\frac{\mu_n}{n}\bigg)-f(x)\bigg]\\
&\quad{}= n^2\bigg[\bigg(f(x)+\frac1n f'(x)+\frac{1}{2n^2}f''(x)+o(n^{-2})\bigg)\frac12\bigg(1+\frac{\mu_n}{n}\bigg)\\
&\qquad\qquad{}+\bigg(f(x)-\frac1n f'(x)+\frac{1}{2n^2}f''(x)+o(n^{-2})\bigg)\frac12\bigg(1-\frac{\mu_n}{n}\bigg)-f(x)\bigg]\\
&\quad{}=\frac12 f''(x)+\mu f'(x)+o(1),
\end{align*}
uniformly over all $x$ with $nx\in{\bf Z}$.  With 
\begin{align*}
\frac12\bigg(1+\frac{\mu_n}{n}\bigg)&=p_0=\frac{\rho^{(1-\alpha)/\alpha}}{1+\rho^{(1-\alpha)/\alpha}}=\frac{(1-\lambda/n)^{(1-\alpha)/\alpha}}{1+(1-\lambda/n)^{(1-\alpha)/\alpha}}\\
&=\frac{1-\lambda(1-\alpha)/(n\alpha)+o(n^{-1})}{2-\lambda(1-\alpha)/(n\alpha)+o(n^{-1})}\\
&=\frac12\bigg(1-\frac{\lambda(1-\alpha)/(n\alpha)+o(n^{-1})}{2-\lambda(1-\alpha)/(n\alpha)+o(n^{-1})}\bigg)\\
&=\frac12\bigg(1-\frac{\lambda(1-\alpha)}{2n\alpha}+o(n^{-1})\bigg), 
\end{align*}
we find that $\mu=-\lambda(1-\alpha)/(2\alpha)$.  And with 
$$
\frac12\bigg(1+\frac{\mu_n}{n}\bigg)=p_1=\frac{1}{1+\rho}=\frac{1}{2-\lambda/n}=\frac12\bigg(\frac{1}{1-\lambda/(2n)}\bigg)=\frac12\bigg(1+\frac{\lambda}{2n}+o(n^{-1})\bigg), 
$$
we find that $\mu=\lambda/2$.  This suffices to complete the proof.  (We leave it to the reader to check that the compact containment condition is satisfied.)
\end{proof}

We assume now that the time parameters $\tau_1>0$ and $\tau_2>0$ of the flashing Brownian ratchet are rational.  Let $m$ be the smallest positive integer such that $m^2\tau_1$ and $m^2\tau_2$ are integers.

\begin{theorem}\label{theorem2}
For $n=m,2m,3m\ldots$, let $\{Y_n(k),\,k=0,1,\ldots\}$ denote the time-inhomogeneous random walk on ${\bf Z}$ that evolves as the simple symmetric random walk for $n^2\tau_1$ steps, then as the random walk of Theorem~\ref{theorem1} for $n^2\tau_2$ steps, then as the simple symmetric random walk for $n^2\tau_1$ steps, then as the random walk of Theorem~\ref{theorem1} for $n^2\tau_2$ steps, and so on.  Let $Y_t$ denote the flashing Brownian ratchet with parameters $\gamma=\lambda(1-\alpha)/2$, $\tau_1$, and $\tau_2$.  If $Y_n(0)/n$ converges in distribution to $Y_0$ as $n\to\infty$, then $\{Y_n(\lfloor n^2 t\rfloor)/n,\,t\ge0\}$ converges in distribution in $D_{\bf R}[0,\infty)$ to $\{Y_t,\, t\ge0\}$ as $n\to\infty$.
\end{theorem}

\begin{proof}
The proof is simply a matter of combining Theorem~\ref{theorem1} with Donsker's theorem applied to the simple symmetric random walk.  The assumption about $m$ ensures that the times $n^2\tau_1$ and $n^2\tau_2$ are integers.
\end{proof}

\section{Modeling Figure~\ref{HATP}, starting at 0}\label{modeling1}

To model Figure~\ref{HATP} accurately, some measurements are needed.  We begin with a cropped \texttt{.pdf} version of the figure and enlarge it on the computer screen to 800\% of normal.  It appears that the figure is rasterized, so our precision is limited.  We measure that $L=206\,$mm and $\alpha L=52\,$mm.  Thus, we imagine that $\alpha=1/4$ was intended, and either the drawing or the measurements of it are slightly in error.  We also measure the height of the normal curve at three places, namely 0, 1, and $-3$, assuming $\alpha=1/4$ and $L=4$.  We
 measure the respective heights to be 99.5\,mm, 81\,mm, and 15\,mm.  Theoretically, the three heights are $(2\pi t)^{-1/2}$, $(2\pi t)^{-1/2}e^{-1/(2t)}$, and $(2\pi t)^{-1/2}e^{-9/(2t)}$.  Therefore, we need to find $t$ such that
$$
99.5e^{-1/(2t)}=81, \qquad 99.5e^{-9/(2t)}=15.
$$
The equations have solutions $t=2.43062$ and $t=2.37830$, respectively.  Because of the crudeness of our measurements, we round off to $t=2.4$.

We conclude that the flashing Brownian ratchet described in Figure~\ref{HATP} evolves as a Brownian motion (starting at 0) for time $\tau_1=2.4$.  Then the Brownian ratchet with $\alpha=1/4$, $L=4$, and $\gamma$ to be specified runs (starting from where the Brownian motion ended) for time $\tau_2$ to be specified.  There is no good way to estimate $\gamma$ and $\tau_2$ from Figure~\ref{HATP}.  We take $\tau_2=\tau_1=2.4$ for convenience and let $\gamma=\lambda(1-\alpha)/2=3\lambda/8$ for several choices of $\lambda$ ($\lambda=1,2,3,4,5$).  Then the Brownian motion runs (starting from where the Brownian ratchet ended) for time $\tau_1=2.4$, then the Brownian ratchet runs for time $\tau_2=2.4$, and so on.  We are interested in the distribution of the process at time $\tau_1+\tau_2=4.8$, which we can compare with the third panel in Figure~\ref{HATP}.

There is no known analytical formula for the density of the flashing Brownian ratchet at time $\tau_1+\tau_2$, but we can approximate it numerically as suggested in Theorem~\ref{theorem2}.  The positive integer $m$ of Theorem~\ref{theorem2} is 5.  In each case we take $n=100$, meaning that at time $\tau_1+\tau_2=4.8$, the approximating random walk has made $4.8n^2=48{,}000$ steps.  We compute its distribution recursively after 1 step, 2 steps, \dots, 48{,}000 steps, using the simple symmetric random walk for the first $24{,}000$ steps, then the asymmetric random walk for the next $24{,}000$ steps.  Notice that the distribution of the random walk after $2k$ steps is concentrated on $\{-2k,-2(k-1),\ldots,0,\ldots,2(k-1),2k\}$, whereas the distribution after $2k+1$ steps is concentrated on $\{-2k-1,-2k+1,\ldots,-1,1,\ldots,2k-1,2k+1\}$.
We save the distribution after $48{,}000$ steps, plot the histogram, and interpolate linearly.  The results are displayed in Figure~\ref{fBr0-1-5}.

\begin{figure}
\centering
\includegraphics[width=3in]{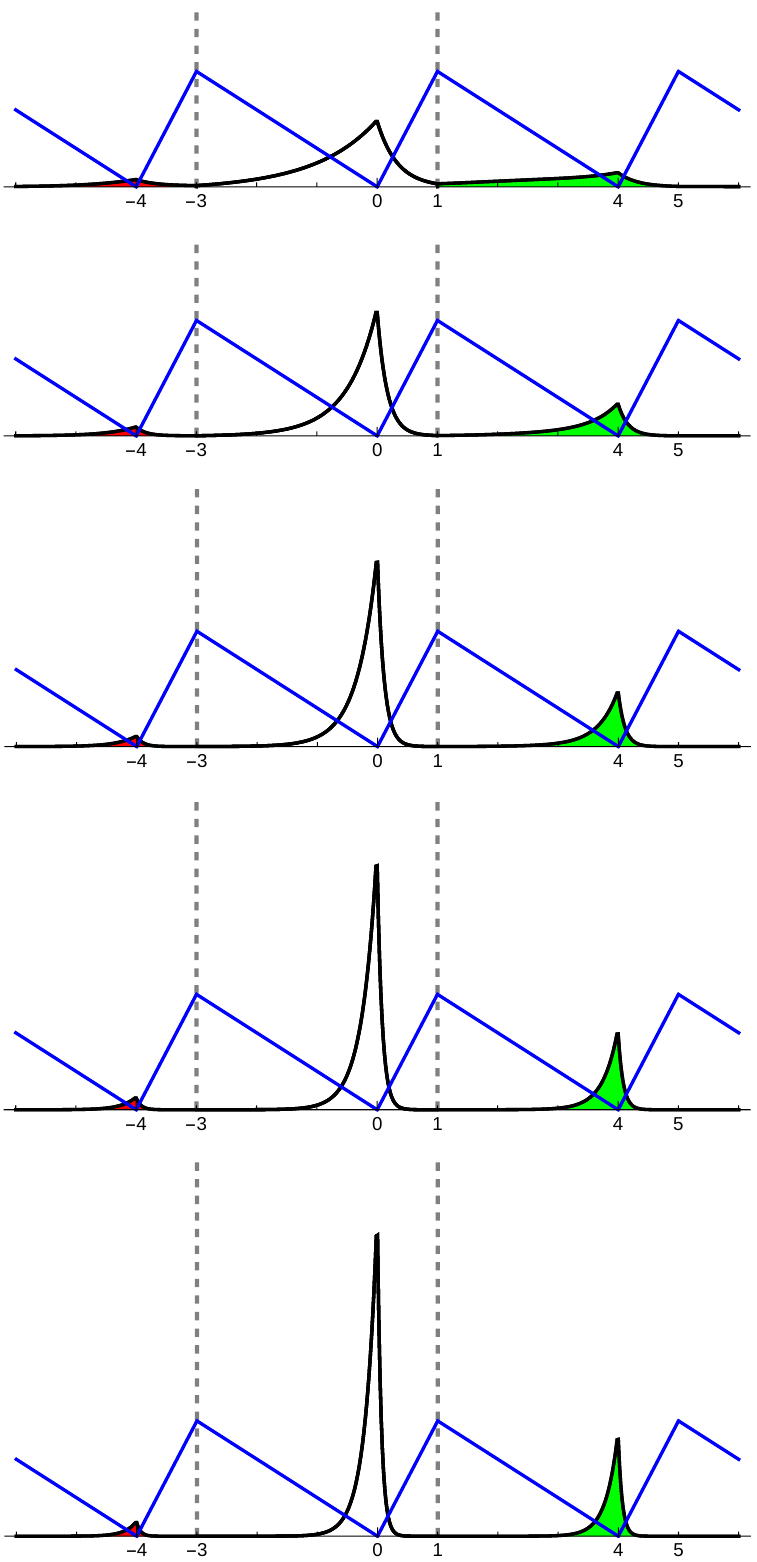}
\caption{\label{fBr0-1-5}We mathematically model the third panel of Figure~\ref{HATP}.  A Brownian motion, starting at 0, runs for time $\tau_1=2.4$.  Then, starting from where the Brownian motion ended, a Brownian ratchet with $\alpha=1/4$, $L=4$, $\gamma=3\lambda/8$, and $\lambda=1,2,3,4,5$ (from top to bottom) runs for time $\tau_2=2.4$.  The black curve is an approximation to the density of the process at time $\tau_1+\tau_2$, via the random walk approximation with $n=100$. The blue curve is the sawtooth potential.}
\end{figure}

There are several notable differences between the figures of Figure~\ref{fBr0-1-5} and the third panel of Figure~\ref{HATP}.  First, the three peaks of the density are pointed, unlike a normal density, so Figure~\ref{Parrondo-Dinis} is more accurate in this regard.  Second, they are asymmetric, with more mass to the left than to the right of $-4$, 0, and 4.  Presumably, the explanation for this is that, for example, the drift to the left on $[0,1)$ is stronger than the drift to the right on $[-3,0)$.  Another distinction is that the ratio of the height of the highest peak to that of the second highest is at least 3 in Figure~\ref{fBr0-1-5} (see Table~\ref{computations-lambda}) but less than 1.5 in Figure~\ref{HATP}.  While this is true for each $\lambda=1,2,3,4,5$, it may be partly a consequence of our arbitrary choice of $\tau_2$.

Consider the case $\lambda=5$.  The areas under the three peaks of the density are respectively 0.0330104, 0.731102, 0.235888.  (These numbers are exact, not for the flashing Brownian ratchet, but for our random walk approximation to it, with $n=100$.)  If the peaks were symmetric, the mean displacement would be $(-4)(0.0330104)+(0)(0.731102)+(4)(0.235888)=0.811510$, but in fact the mean displacement is 0.678364 (again, an approximation) because of the asymmetry of each peak.

Table~\ref{computations-lambda} shows the effect of varying $\lambda$ on several statistics of interest.

\begin{table}[htb]
\caption{\label{computations-lambda}Computations for the $n$th random walk ($n=100$) approximating the flashing Brownian ratchet with $\alpha=1/4$, $L=4$, $\gamma=3\lambda/8$, $\tau_1=\tau_2=2.4$, and initial state 0, at time $\tau_1+\tau_2$, illustrating the effect of varying the strength $\gamma$ of the drift of the Brownian ratchet.\medskip}
\catcode`@=\active \def@{\hphantom{0}}
\begin{center}
\begin{footnotesize}
\begin{tabular}{cccc}
\noalign{\smallskip}
\hline
\noalign{\smallskip}
$\lambda$   &  areas of the three peaks   &  heights of the three peaks & mean \\
&&& displacement\\
\noalign{\smallskip}
\hline
\noalign{\smallskip}
@1 & $(0.0688267,0.701114,0.230060)$  &  $(0.0627471,0.566531,0.121751)$ & 0.0595931 \\
@2 & $(0.0500629,0.734941,0.214996)$  &  @$(0.0756255,1.06860,0.274227)$ & 0.297582@ \\
@3 & $(0.0400379,0.737033,0.222929)$  &  @$(0.0875995,1.59779,0.464698)$ & 0.496585@ \\
@4 & $(0.0354116,0.734036,0.230552)$  &  @$(0.1021090,2.11341,0.657213)$ & 0.611651@ \\
@5 & $(0.0330104,0.731102,0.235888)$  &  @@$(0.117836,2.60974,0.839352)$ & 0.678364@ \\
10 & $(0.0290537,0.723174,0.247772)$  &  @@$(0.197900,4.92657,1.68412)$@ & 0.809036@ \\
15 & $(0.0279536,0.719952,0.252094)$  &  @@$(0.273152,7.03601,2.45801)$@ & 0.853220@ \\
20 & $(0.0274363,0.718221,0.254343)$  &  @@$(0.342844,8.97601,3.17124)$@ & 0.875658@ \\
25 & $(0.0271326,0.717131,0.255736)$  &  @@$(0.407788,10.7794,3.83499)$@ & 0.889397@ \\
50 & $(0.0264993,0.714662,0.258839)$  &  @@$(0.695524,18.7599,6.77822)$@ & 0.919557@ \\
\noalign{\smallskip}
\hline
\end{tabular}
\end{footnotesize}
\end{center}
\end{table}

We might ask whether, as suggested in Figures~\ref{HATP} and \ref{Parrondo-Dinis}, the areas of the three peaks are equal to the corresponding areas under the normal curve.  The latter areas are 
$$
\Phi\bigg(\frac{-3}{\sigma}\bigg)=0.0264038,\quad \Phi\bigg(\frac{1}{\sigma}\bigg)-\Phi\bigg(\frac{-3}{\sigma}\bigg)=0.714294,\quad 1-\Phi\bigg(\frac{1}{\sigma}\bigg)=0.259303,
$$
where $\sigma=\sqrt{2.4}$.  It seems evident that this is true in the limit as $\lambda\to\infty$.  See Table~\ref{computations-lambda}.

We return to the case $\lambda=5$.  To get a sense of the rate of convergence in Theorem~\ref{theorem2}, we provide in Table~\ref{computations-n} computed values of several statistics as functions of $n$ $(=10, 20, 30, \ldots, 200)$.  

\begin{table}[hbt]
\caption{\label{computations-n}Computations for the $n$th random walk approximating the flashing Brownian ratchet with $\alpha=1/4$, $L=4$, $\gamma=3\lambda/8$, $\lambda=5$, $\tau_1=\tau_2=2.4$, and initial state 0, at time $\tau_1+\tau_2$, illustrating the rate of convergence in a special case of Theorem~\ref{theorem2}.\medskip}
\catcode`@=\active \def@{\hphantom{0}}
\begin{center}
\begin{footnotesize}
\begin{tabular}{cccc}
\noalign{\smallskip}
\hline
\noalign{\smallskip}
$n$   &  areas of the three peaks   &  heights of the three peaks & mean \\
&&& displacement\\
\noalign{\smallskip}
\hline
\noalign{\smallskip}
@10 & $(0.0279285,0.716249,0.255823)$  &  $(0.0733035,1.88015,0.669941)$ & 0.791225 \\
@20 & $(0.0309972,0.725965,0.243038)$  &  $(0.0931706,2.18234,0.728788)$ & 0.713194 \\
@30 & $(0.0318689,0.728297,0.239835)$  &  @$(0.102331,2.33873,0.768103)$ & 0.696690 \\
@40 & $(0.0322853,0.729350,0.238365)$  &  @$(0.107491,2.42843,0.791417)$ & 0.689617 \\
@50 & $(0.0325301,0.729952,0.237518)$  &  @$(0.110785,2.48602,0.806553)$ & 0.685678 \\
@60 & $(0.0326914,0.730343,0.236965)$  &  @$(0.113066,2.52599,0.817114)$ & 0.683162 \\
@70 & $(0.0328059,0.730618,0.236576)$  &  @$(0.114737,2.55531,0.824886)$ & 0.681414 \\
@80 & $(0.0328913,0.730821,0.236288)$  &  @$(0.116014,2.57774,0.830840)$ & 0.680129 \\
@90 & $(0.0329575,0.730978,0.236065)$  &  @$(0.117021,2.59543,0.835543)$ & 0.679144 \\
100 & $(0.0330104,0.731102,0.235888)$  &  @$(0.117836,2.60974,0.839352)$ & 0.678364 \\
110 & $(0.0330535,0.731203,0.235743)$  &  @$(0.118508,2.62156,0.842498)$ & 0.677731 \\
120 & $(0.0330894,0.731287,0.235623)$  &  @$(0.119073,2.63148,0.845140)$ & 0.677208 \\
130 & $(0.0331197,0.731358,0.235522)$  &  @$(0.119553,2.63993,0.847391)$ & 0.676768 \\
140 & $(0.0331457,0.731419,0.235435)$  &  @$(0.119967,2.64720,0.849331)$ & 0.676392 \\
150 & $(0.0331681,0.731471,0.235361)$  &  @$(0.120327,2.65354,0.851020)$ & 0.676068 \\
160 & $(0.0331878,0.731517,0.235295)$  &  @$(0.120644,2.65910,0.852504)$ & 0.675785 \\
170 & $(0.0332051,0.731557,0.235238)$  &  @$(0.120924,2.66403,0.853818)$ & 0.675537 \\
180 & $(0.0332205,0.731593,0.235187)$  &  @$(0.121174,2.66842,0.854990)$ & 0.675316 \\
190 & $(0.0332343,0.731625,0.235141)$  &  @$(0.121398,2.67236,0.856041)$ & 0.675120 \\
200 & $(0.0332467,0.731653,0.235100)$  &  @$(0.121600,2.67592,0.856990)$ & 0.674943 \\
\noalign{\smallskip}
\hline
\end{tabular}
\end{footnotesize}
\end{center}
\end{table}
\afterpage{\clearpage}

\section{Modeling Figure~\ref{HATP}, starting at stationarity}\label{modeling2}

By properties of diffusion processes with constant diffusion and gradient drift, the Brownian ratchet has an invariant measure $\pi$ of the form
\begin{equation}\label{invariant}
\pi(dx)=C\exp\{-2\gamma V(x)\}\,dx.
\end{equation}
The density of $\pi$ is a periodic function (with period $L$) whose maxima occur at the minima of the sawtooth potential.

\begin{theorem}
The Brownian ratchet with parameters $\alpha$, $L$, and $\gamma$ has an invariant measure $\pi$ of the form \eqref{invariant}.  The wrapped Brownian ratchet with the same parameters has an invariant measure of the same form, restricted to $[0,L)$.
\end{theorem}

\begin{proof}
We use a different characterization of the Brownian ratchet.  We take $\mathscr{D}(\mathscr{L})$, the domain of $\mathscr{L}$, to be the space of real-valued $C^1$ functions $f$ on \textbf{R} with limits at $\pm\infty$ such that $f'$ is absolutely continuous and has a right derivative, denoted by $f''$, with $\mathscr{L}f$ continuous on \textbf{R} with limits at $\pm\infty$.  In particular, the discontinuities of $f''$ must be compatible with those of the drift coefficient $\mu$.  Thus,
\begin{align*}
\frac12 f''(nL-)+\frac{\gamma}{1-\alpha}f'(nL)&=\frac12 f''(nL)-\frac{\gamma}{\alpha}f'(nL),\\
\frac12 f''((n+\alpha)L-)-\frac{\gamma}{\alpha}f'((n+\alpha)L)&=\frac12 f''((n+\alpha)L)+\frac{\gamma}{1-\alpha}f'((n+\alpha)L),
\end{align*}
for all $n\in{\bf Z}$.  Mandl \cite[pp.\ 25, 38]{M68} showed that $\mathscr{L}$ generates a Feller semigroup on $C[-\infty,\infty]$.  Because both boundaries are natural, the Feller semigroup can be restricted to $\widehat C({\bf R})$, the subspace of continuous functions vanishing at $\pm\infty$.  Moreover, the subspace of $\mathscr{D}(\mathscr{L})$ consisting of functions with compact support is a core for the generator.
To confirm \eqref{invariant}, 
\begin{align}\label{invariant-proof}
\int_{\bf R}(\mathscr{L}f)(x)\,\pi(dx)&=\sum_{n=-\infty}^\infty\int_{nL}^{(n+1)L}[\textstyle{\frac12}f''(x)-\gamma V'(x)f'(x)]C\exp\{-2\gamma V(x)\}\,dx\nonumber\\
&=C\sum_{n=-\infty}^\infty\int_{nL}^{(n+1)L}\textstyle{\frac12}[f'(x)\exp\{-2\gamma V(x)\}]'\,dx\nonumber\\
&=(C/2)\sum_{n=-\infty}^\infty[f'((n+1)L)\exp\{-2\gamma V((n+1)L)\}\nonumber\\
&\hspace{4cm}{}-f'(nL)\exp\{-2\gamma V(nL)\}]\nonumber\\
&=(C/2)\sum_{n=-\infty}^\infty[f'((n+1)L)-f'(nL)]\nonumber\\
&=0
\end{align}
for every $f\in\mathscr{D}(\mathscr{L})$ with compact support (the sum telescopes), since $V(nL)=0$ for all $n\in{\bf Z}$.

For the second assertion, we take $\mathscr{D}(\mathscr{L})$ to be the space of real-valued $C^1$ functions $f$ on the circle $[0,L)$ such that $f'$ is absolutely continuous and has a right derivative, denoted by $f''$, with $\mathscr{L}f$ continuous on the circle $[0,L)$.  Thus, $f(0)=f(L-)$, $f'(0)=f'(L-)$,  
\begin{align*}
\frac12 f''(0)-\frac{\gamma}{\alpha}f'(0)&=\frac12 f''(L-)+\frac{\gamma}{1-\alpha}f'(L-),\\
\frac12 f''(\alpha L-)-\frac{\gamma}{\alpha}f'(\alpha L)&=\frac12 f''(\alpha L)+\frac{\gamma}{1-\alpha}f'(\alpha L).
\end{align*}
Finally, $\int_0^L(\mathscr{L}f)(x)\,\pi(dx)=0$ as in \eqref{invariant-proof} except with the sums over $n$ replaced by their $n=0$ terms.
\end{proof}

For both invariant measures (unrestricted and restricted) we expect there is a uniqueness result but we currently lack a proof.  Notice that the mean drift, with respect to the invariant probability measure, is equal to 
\begin{align*}
\frac{\int_0^L \mu(x)\,\exp\{-2\gamma V(x)\}\,dx}{\int_0^L \exp\{-2\gamma V(x)\}\,dx}&=\frac{-\gamma\int_0^L V'(x)\,\exp\{-2\gamma V(x)\}\,dx}{\int_0^L \exp\{-2\gamma V(x)\}\,dx}\\
&=\frac{\int_0^L [\exp\{-2\gamma V(x)\}]'\,dx}{2\int_0^L \exp\{-2\gamma V(x)\}\,dx}\\
&=\frac{\exp\{-2\gamma V(L)\}-\exp\{-2\gamma V(0)\}}{2\int_0^L \exp\{-2\gamma V(x)\}\,dx}\\
&=0
\end{align*}
since $V(L)=V(0)=0$.  Thus, the mean drift is 0. 

Denote the flashing Brownian ratchet at time $t$, starting from $x\in{\bf R}$ at time 0, by $Y_t^x$, and the wrapped flashing Brownian ratchet at time $t$, starting from $x\in[0,L)$ at time 0, by $\bar Y_t^x$.  Then the one-step transition function
\begin{equation}\label{transition-function}
\bar P(x,\cdot):=\P(\bar Y_{\tau_1+\tau_2}^x\in\cdot)
\end{equation}
for a continuous-state Markov chain has a stationary distribution $\bar\pi$.  (Existence is automatic from the Feller property and the compactness of the state space; recall that the endpoints of $[0,L)$ are identified.  Nevertheless, no analytical formula is known, and uniqueness is expected but unproved.) The mean displacement $\bar\mu$ of the flashing Brownian ratchet over the time interval $[0,\tau_1+\tau_2]$, starting from the stationary distribution $\bar\pi$, namely
\begin{align*}
\bar\mu&:=\int_0^L\E[Y_{\tau_1+\tau_2}^x-Y_0^x]\,\bar\pi(dx)\\
&\phantom{:}=\int_{-(1-\alpha)L}^{\alpha L}\E[Y_{\tau_1+\tau_2}^x-Y_0^x]\,\bar\pi(dx)
\end{align*}
is a statistic of primary interest.  The second equality is a consequence of the periodicity of the integrand (with period $L$) and the convention that we do not distinguish notationally between $\bar\pi$ and its image under the mapping
$$
x\mapsto\begin{cases}x&\text{if $0\le x<\alpha L$,}\\ x-L&\text{if $\alpha L\le x< L$.}\end{cases}
$$
The advantage of modifying $\bar\pi$ in this way is that, when regarded as a measure on ${\bf R}$, it becomes unimodal instead of U-shaped.

We propose to approximate $\bar\mu$ as follows.  The integrand can be estimated as before, the only difference being that the starting point of the flashing Brownian ratchet is $x$, not 0.  The stationary distribution $\bar\pi$ of the one-step transition function \eqref{transition-function} can be approximated by the stationary distribution of the finite Markov chain whose one-step transition matrix has the form 
$$
P(i,j):=\P(\bar Y_n(n^2(\tau_1+\tau_2))=j\mid \bar Y_n(0)=i),\quad i,j=0,1,\ldots,Ln-1.
$$
A small technical issue, if $Ln$ is even,  is that this Markov chain fails to be irreducible if $n^2(\tau_1+\tau_2)$ is even, in which case we replace it by $n^2(\tau_1+\tau_2)+1$.  Then it becomes irreducible and there is a unique stationary distribution.  The black curve of the first panel of Figure~\ref{fBrstat-fig} is an approximation to the density of $\bar\pi$ with support $[-3,1)$ instead of $[0,4)$. Starting from the approximate $\bar\pi$ at time 0, the second and third panels show the approximations to the density at times $\tau_1$ and $\tau_1+\tau_2$, respectively. Figure~\ref{fBrstat-fig} can be regarded as a more accurate version of Figures~\ref{HATP} and \ref{Parrondo-Dinis}.  Computations show that
\begin{equation}\label{mubar}
\bar\mu=0.684827,
\end{equation}
which is slightly larger than the corresponding number in Table~\ref{computations-lambda}. 

\begin{figure}[htb]
\centering
\includegraphics[width=3.58in]{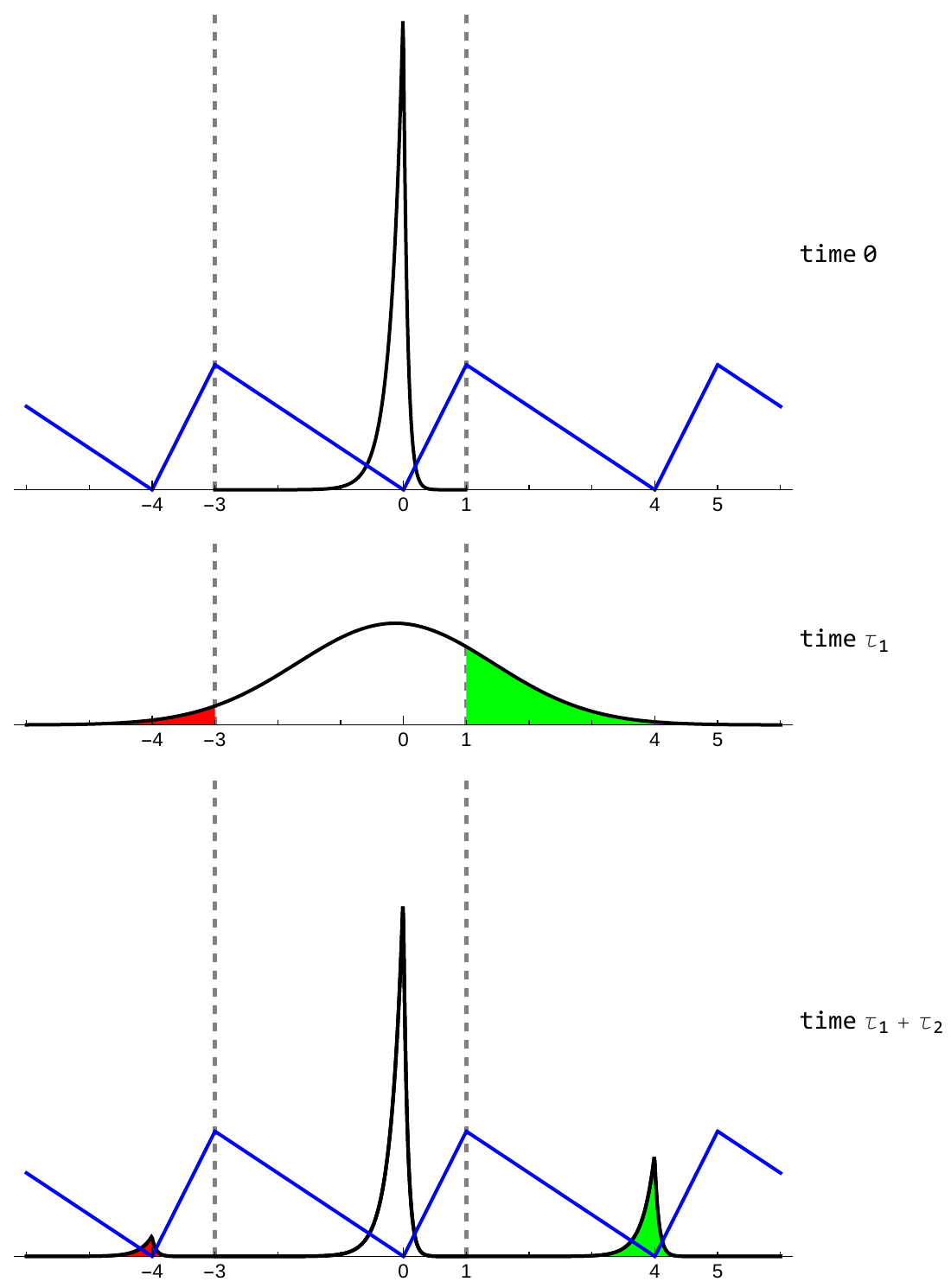}
\caption{\label{fBrstat-fig}We mathematically model the full Figure~\ref{HATP}.  Starting from the stationary distribution $\bar\pi$ with support $[-3,1)$, a Brownian motion runs for time $\tau_1=2.4$.  Then, starting from where the Brownian motion ended, a Brownian ratchet with $\alpha=1/4$, $L=4$, $\gamma=3\lambda/8$, and $\lambda=5$, runs for time $\tau_2=2.4$.  The black curves in all three panels are based on the random walk approximation with $n=100$.  The blue curves represent the sawtooth potential.  The vertical axes in the first and third panels are comparable, whereas the vertical axis in the second panel has been stretched for clarity.}
\end{figure}
\afterpage{\clearpage}

Because the evaluation in \eqref{mubar} is computationally intensive, it would be a challenging numerical optimization problem to determine the values of $\tau_1$ and $\tau_2$ that maximize the long-term mean displacement per unit time, $\bar\mu/(\tau_1+\tau_2)$.  They would depend on $\alpha$, $L$, and $\gamma$.

A problem that we hope to address in the near future is to establish a strong law of large numbers for flashing Brownian ratchet increments, perhaps analogous to our earlier SLLN (Ethier and Lee \cite{EL09}) for the sequence of Parrondo-game profits.

\section*{Acknowledgments}

The work of S. N. Ethier was partially supported by a grant from the Simons Foundation (429675). 
The work of J. Lee was supported by Basic Science Research Program through the National Research Foundation of Korea (NRF) funded by the Ministry of Science, ICT \& Future Planning (NRF-2017R1A2B1007089).

\end{document}